\documentclass{amsart}

\newtheorem{theorem}{Theorem}[section]
\newtheorem{corollary}{Corollary}[section]

\newtheorem{remark}{Remark}[section]

\input{tcilatex}

\begin{document}
\title[Biharmonic maps...]{Biharmonic maps between doubly warped product
manifolds}
\date{}
\author{Selcen Y\"{U}KSEL PERKTA\c{S}, Erol KILI\c{C}}
\address{Correspondence Address: Inonu University, Faculty of Arts and
Sciences, Department of Mathematics, 44280 Malatya/TURKEY}
\email{selcenyuksel@inonu.edu.tr, ekilic@inonu.edu.tr}

\begin{abstract}
In this paper biharmonic maps between doubly warped product manifolds are
studied. We show that the inclusion maps of Riemannian manifolds $B$ and $F$
into the doubly warped product $_{f}B\times _{b}F$ can not be proper
biharmonic maps. Also we analyze the conditions for the biharmonicity of
projections $_{f}B\times_{b}F\rightarrow B$ and $_{f}B\times_{b}F\rightarrow
F$ , respectively. Some characterizations for non-harmonic biharmonic maps
are given by using product of harmonic maps and warping metric. Specially,
in the case of $f=1$, the results for warped product in \cite{Balmus-mont}
are obtained.
\end{abstract}

\maketitle

\section{\protect\bigskip Introduction}

The study of biharmonic maps between Riemannian manifolds, as a
generalization of harmonic maps, was suggested by J. Eells and J. H. Sampson
in \cite{Eells}. The energy of a smooth map $\varphi :(B,g_{B})\rightarrow
(F,g_{F})$ between two Riemannian manifolds is defined by $E(\varphi )=\frac{%
1}{2}\int_{D}|d\varphi |^{2}v_{g_{B}}$ and $\varphi $ is called harmonic if
it's a critical point of energy. From the first variation formula for the
energy, the Euler-Lagrange equation associated to the energy is given by $%
\tau (\varphi )=0$ where $\tau (\varphi )=trace$ $\nabla d\varphi $ is the
tension field of $\varphi $ (see also [2, 8, 11]).

The bienergy functional $E_{2}$ of a smooth map $\varphi
:(B,g_{B})\rightarrow (F,g_{F})$ is defined by integrating the square norm
of the tension field, $E_{2}(\varphi )=\frac{1}{2}\int_{D}|\tau (\varphi
)|^{2}v_{g_{B}}.$ The first variation formula for the bienergy, derived in
\cite{Jiang,Jiang2} shows that the Euler-Lagrange equation for $E_{2}$ is%
\begin{equation*}
\tau _{2}(\varphi )=-J_{\varphi }(\tau (\varphi ))=-\Delta \tau (\varphi
)-traceR^{F}(d\varphi ,\tau (\varphi ))d\varphi=0
\end{equation*}%
where $J^{\varphi }$ is formally the Jacobi operator of $\varphi .$ Since
any harmonic map is biharmonic, we are interested in non-harmonic biharmonic
maps which are called proper biharmonic, (see also [15,16]).

In \cite{Baird} P. Baird and D. Kamissoko constructed new examples of proper
biharmonic maps between Riemannian manifolds by firstly taking a harmonic
map $\varphi :B\rightarrow F$ which is automatically biharmonic and then
deforming the metric conformally on $B$ to render $\varphi $ biharmonic,
(see also [3]). In \cite{Balmus-mont} the authors studied the biharmonic
maps between warped product manifolds. In this paper biharmonicity of the
iclusion $i:F\rightarrow B\times _{b}F$ of a Riemannian manifold $F$ into
the warped product manifold $B\times _{b}F$ and of the projection from $%
B\times _{b}F$ into the first factor are investigated. Also the authors gave
in \cite{Balmus-mont} two new classes of proper biharmonic maps by using
product of harmonic maps and warping the metric in the domain or codomain.

In this paper, we analyze the behaviour of the biharmonic equation by taking
into account to doubly warped products. Warped products were first defined
by O'Neill and Bishop in 1969, see \cite{Bishop}. By using this concept they
constructed Riemannian manifolds with negative sectional curvature. Also in
\cite{O'Neill} O'Neill gave the curvature formulas of warped products in the
terms of curvatures of components of warped products and studied
Robertson-Walker, static, Schwarzchild and Kruskal space-times as warped
products. In general doubly warped products can be considered as a
generalization of singly warped products or simply warped products. A doubly
warped product manifold is a product manifold $B\times F$ of two Riemannian
manifolds $(B,g_{B})\ $\ and $(F,g_{F})$ endowed with the metric $%
g=f^{2}g_{B}\oplus b^{2}g_{F}$ \ where $b:B\rightarrow (0,\infty )$ and $%
f:F\rightarrow (0,\infty )$ are smooth functions. The canonical leaves $%
\{x_{0}\}\times F$ and $B\times \{y_{0}\}$ of a doubly warped product
manifold $_{f}B\times _{b}F$ are totally umbilic submanifolds, which
intersect perpendicularly \cite{Reckziegel}, (see also [7, 10, 18, 19]).
When $f=1$, $_{1}B\times _{b}F$ becomes a warped product manifold and in
this case the leaves $B\times \{y_{0}\}$ are totally geodesic.

This article organized as follow.

In the first and second sections we give some basic definitions on
biharmonic maps and doubly warped product manifolds, respectively. In the
case of warped products since $B\times \{y_{0}\}$ is totally geodesic so
biharmonic, the authors in \cite{Balmus-mont} investigated only the
biharmonicity of the inclusion of the Riemannian manifold $F$ into the
warped product $B\times _{b}F$. In section 3, by considering the situation
of doubly warped product as a generalization of warped products, we analyze
the conditions for both of the leaves $\{x_{0}\}\times F$ and $B\times
\{y_{0}\}$ to be biharmonic as a submanifold and we show that both of the
leaves $\{x_{0}\}\times F$ and $B\times \{y_{0}\}$ can not be proper
biharmonic as a submanifold of the doubly warped product manifold $%
_{f}B\times _{b}F$. The product of two harmonic maps is clearly harmonic. If
the metric in the domain or codomain is deformed conformally, then the
harmonicity is lost. Then it's possible to define proper biharmonic maps
using products of two harmonic maps. In the next section we find some
results on product maps to be proper biharmonic.

\section{Biharmonic maps between Riemannian manifolds}

\setcounter{equation}{0} \renewcommand{\theequation}{2.\arabic{equation}}

Let $(B,g_{B})$ and $(F,g_{F})$ be Riemann manifolds and $\varphi
:(B,g_{B})\rightarrow (F,g_{F})$ be a smooth map. The tension field of $%
\varphi $ is given by
\begin{eqnarray}
\tau (\varphi )=trace\nabla d\varphi,
\end{eqnarray}%
where $\nabla d\varphi $ is the second fundamental form of $\varphi $.

Biharmonic maps $\varphi :(B,g_{B})\rightarrow (F,g_{F})$ between Riemannian
manifolds are critical points of the bienergy functional
\begin{eqnarray}
E_{2}(\varphi )=\frac{1}{2}\int_{D}|\tau (\varphi )|^{2}v_{g_{B}},
\end{eqnarray}
for any compact domain $D\subset B$. Biharmonic maps are a natural
generalization of the well-known harmonic maps, the extremal points of the
energy functional defined by
\begin{eqnarray}
E(\varphi )=\frac{1}{2}\int_{D}|d\varphi |^{2}v_{g_{B}.}
\end{eqnarray}
The Euler-Lagrange equation for the energy is $\tau (\varphi )=0$.

The first variation formula of $E_{2}(\varphi )$ is
\begin{eqnarray}
\frac{\partial }{\partial t}E_{2}(\varphi _{t})|_{t=0}=-\int_{D}<J_{\varphi
}(\tau (\varphi )),w>v_{g_{B}},
\end{eqnarray}
where $w=\frac{\partial \varphi }{\partial t}|_{t=0}$ is the variational
vector field of the variation $\{\varphi _{t}\}$ of $\varphi $. The
Euler-Lagrange equation corresponding to $E_{2}(\varphi )$ is given by the
vanishing of the bitension field
\begin{eqnarray}
\tau _{2}(\varphi )=-J_{\varphi }(\tau (\varphi ))=-\Delta \tau (\varphi
)-traceR^{F}(d\varphi ,\tau (\varphi ))d\varphi,
\end{eqnarray}
where $J^{\varphi }$ is the Jacobi operator of $\varphi $. Here $\Delta $ is
the rough Laplacian on sections of the pull-back bundle $\varphi ^{-1}(TF)$
defined by, for an orthonormal frame field $\{B_{j}\}_{j=1}^{m}$ on $B$,
\begin{eqnarray}
\Delta v &=&-trace_{g_{b}}(\nabla ^{\varphi })^{2}v  \notag \\
&=&-\sum_{j=1}^{m}\{\nabla _{B_{j}}^{\varphi }\nabla _{B_{j}}^{\varphi
}v-\nabla _{\nabla _{B_{j}}^{B}B_{j}}^{\varphi }v\},\,\,\,v\in \Gamma
(\varphi ^{-1}(TF)),
\end{eqnarray}
with $\nabla ^{\varphi }$ is representing the connection in the pull-back
bundle $\varphi ^{-1}(TF)$ and $\nabla ^{B}$ is the Levi-Civita connection
on $M$ and $R^{F}$ is the curvature operator
\begin{eqnarray}
R^{F}(X,Y)Z=\nabla _{X}^{F}\nabla _{Y}^{F}Z-\nabla _{Y}^{F}\nabla
_{X}^{F}Z-\nabla _{\lbrack X,Y]}^{F}Z.
\end{eqnarray}
Clearly any harmonic map is biharmonic. We call the non-harmonic biharmonic
maps proper biharmonic maps .

\section{ Doubly warped product manifolds}

\setcounter{equation}{0} \renewcommand{\theequation}{3.\arabic{equation}}

Let $(B,g_{B})$ and $(F,g_{F})$ be Riemannian manifolds of dimensions $m$
and $n$, respectively and let $b:B\rightarrow (0,\infty )$ and $%
f:F\rightarrow (0,\infty )$ be smooth functions. As a generalization of the
warped product of two Riemannian manifolds, a doubly warped product of
Riemannian manifolds $(B,g_{B})$ and $(F,g_{F})$ with warping functions $b$
and $f$ is a product manifold $B\times F$ with metric tensor
\begin{eqnarray}
g=f^{2}g_{B}\oplus b^{2}g_{F},
\end{eqnarray}
given by
\begin{eqnarray}
g(X,Y)=(f\circ \sigma )^{2}g_{B}(d\pi (X),d\pi (Y))+(b\circ \pi
)^{2}g_{F}(d\sigma (X),d\sigma (Y)),
\end{eqnarray}
where $X,Y\in \Gamma (T(B\times F))$ and $\pi :B\times F\rightarrow B$ and $%
\sigma :B\times F\rightarrow F$ are the canonical projections. We denote the
doubly warped product of Riemannian manifolds $(B,g_{B})\ $\ and $(F,g_{F})$
by $_{f}B\times _{b}F$. If $f=1$ then $_{1}B\times _{b}F=B\times _{b}F$
becomes a warped product of Riemannian manifolds $B$ and $F$.

Let $(B,g_{B})$ and $(F,g_{F})$ be Riemannian manifolds with Levi-Civita
connections $\nabla ^{B}$ and $\nabla ^{F}$, respectively and let $\nabla $
and $\overline{\nabla }$ denote the Levi-civita connections of the product
manifold $B\times F$ and doubly warped product manifold $_{f}B\times _{b}F$,
respectively. The Levi-Civita connection of doubly warped product manifold $%
_{f}B\times _{b}F$ is defined by
\begin{eqnarray}
\overline{\nabla }_{X}Y &=&\nabla _{X}Y+\frac{1}{2b^{2}}%
X_{1}(b^{2})(0,Y_{2})+\frac{1}{2b^{2}}Y_{1}(b^{2})(0,X_{2})  \notag \\
&&+\frac{1}{2f^{2}}X_{2}(f^{2})(Y_{1},0)+\frac{1}{2f^{2}}%
Y_{2}(f^{2})(X_{1},0)  \notag \\
&&-\frac{1}{2}g_{B}(X_{1},Y_{1})({grad}\ f^{2},0)-\frac{1}{2}%
g_{F}(X_{2},Y_{2})(0,{grad}\ b^{2})  \label{eq:3.3}
\end{eqnarray}%
for any $X,Y\in \Gamma (T(B\times F))$, where $X=(X_{1},X_{2})$, $%
Y=(Y_{1},Y_{2})$, $X_{1},Y_{1}\in \Gamma (TB)$ and\ $X_{2},Y_{2}\in \Gamma
(TF)$.

If $R$ and $\overline{R}$ denote the curvature tensors of $B\times F$ and $%
_{f}B\times _{b}F$, respectively then we have the following relation:
\begin{eqnarray}
&&\bar{R}(X,Y)-R(X,Y)=  \notag \\
&&\frac{1}{2b^{2}}\{[((\nabla _{Y_{1}}^{B}\func{grad}b^{2}-\frac{1}{2b^{2}}%
Y_{1}(b^{2})\func{grad}b^{2},0)-\frac{1}{2f^{2}}(0,Y_{1}(b^{2})\func{grad}%
f^{2}))\wedge _{g}(0,X_{2})  \notag \\
&&-((\nabla _{X_{1}}^{B}\func{grad}b^{2}-\frac{1}{2b^{2}}X_{1}(b^{2})\func{%
grad}b^{2},0)-\frac{1}{2f^{2}}(0,X_{1}(b^{2})\func{grad}f^{2}))\wedge
_{g}(0,Y_{2})]  \notag \\
&&+\frac{1}{2b^{2}}|\func{grad}b^{2}|^{2}(0,X_{2})\wedge _{g}(0,Y_{2})\}
\notag \\
&&+\frac{1}{2f^{2}}\{[((0,\nabla _{Y_{2}}^{F}\func{grad}f^{2}-\frac{1}{2f^{2}%
}Y_{2}(f^{2})\func{grad}f^{2})-\frac{1}{2b^{2}}(Y_{2}(f^{2})\func{grad}%
b^{2},0)\wedge _{g}(X_{1},0)  \notag \\
&&-((0,\nabla _{X_{2}}^{F}\func{grad}f^{2}-\frac{1}{2f^{2}}X_{2}(f^{2})\func{%
grad}f^{2})-\frac{1}{2b^{2}}(X_{2}(f^{2})\func{grad}b^{2},0)\wedge
_{g}(Y_{1},0)]  \notag \\
&&{\text \ \ \ \ \ \ \ \ \ }\, \, \, \, \,+\frac{1}{2f^{2}}|\func{grad}%
f^{2}|^{2}(X_{1},0)\wedge _{g}(Y_{1},0)\},
\end{eqnarray}
where the wedge product $X\wedge _{g}Y$ denotes the linear map $Z\rightarrow
g(Y,Z)X-g(X,Z)Y$ for all $X,$ $Y,$ $Z\in \Gamma (T(B\times F))$ .

\section{Biharmonicity of the inclusion maps}

\setcounter{equation}{0} \renewcommand{\theequation}{4.\arabic{equation}}

Let $(_{f}B\times _{b}F,g)$ be a doubly warped product manifold. For $%
y_{0}\in F$, let us consider the inclusion map of $B$
\begin{eqnarray*}
i_{y_{0}}:(B,g_{B}) &\rightarrow &(_{f}B\times _{b}F,g) \\
x &\rightarrow &(x,y_{0})
\end{eqnarray*}%
at the point $y_{0}$ level in $_{f}B\times _{b}F$ and for $x_{0}\in B$ \ let
\begin{eqnarray*}
i_{x_{0}} :(F,g_{F})&\rightarrow& (_{f}B\times _{b}F,g) \\
y &\rightarrow &(x_{0},y)
\end{eqnarray*}%
be the inclusion map of $F$ at the point $x_{0}$ level in $_{f}B\times
_{b}F. $ In this section we obtain some non-existence results for the
biharmonicity of inclusion maps $i_{y_{0}}$ of $B$ and $i_{x_{0}}$ of $F.$

\begin{theorem}
The bitension field of the inclusion map $i_{y_{0}}:(B,g_{B})\rightarrow
(_{f}B\times _{b}F,g)$ is given by
\begin{eqnarray*}
\tau _{2}(i_{y_{0}}) &=&\{-\frac{m^{2}}{8b^{2}}|{grad}\ f^{2}|^{2}({grad}\
b^{2},0) \\
&&+\frac{m}{2}\Delta (\ln b)(0,{grad}\ f^{2})+\frac{m^{2}}{8}(0,{grad}(|{grad%
}\ f^{2}|^{2})\}|_{i_{y_{0}}}.
\end{eqnarray*}
\end{theorem}

\begin{proof}
Let $\{B_{j}\}_{j=1}^{m}$ be an orthonormal frame on $(B,g_{B})$. By using
the equation (2.1) we obtain the tension field of $i_{y_{0}}$
\begin{eqnarray*}
\tau (i_{y_{0}}) &=&trace_{g_{B}}\nabla di_{y_{0}} \\
&=&\sum_{j=1}^{m}\{\nabla _{B_{j}}di_{y_{0}}(B_{j})-di_{y_{0}}(\nabla
_{B_{j}}^{B}B_{j})\} \\
&=&\sum_{j=1}^{m}\{\overline{\nabla }_{(B_{j},0)}(B_{j},0)-(\nabla
_{B_{j}}^{B}B_{j},0)\} \\
&=&-\frac{m}{2}(0,{grad}\ f^{2})|_{i_{y_{0}}}
\end{eqnarray*}%
Here it's obvious from the expression of the tension field of $i_{y_{0}}$
that $\ i_{y_{0}}$ is harmonic if and only if $({grad}\
f^{2})|_{i_{y_{0}}}=0 $.\newline
Now to get the bitension field of $i_{y_{0}}:(B,g_{B})\rightarrow
(_{f}B\times _{b}F,g)$, firstly let us compute the rough Laplacian of the
tension field of $\tau (i_{y_{0}})$. We have
\begin{eqnarray*}
\nabla _{B_{j}}\tau (i_{y_{0}}) &=&-\frac{m}{2}\nabla _{B_{j}}(0,{grad}\
f^{2})|_{i_{y_{0}}} \\
&=&-\frac{m}{2}(\overline{\nabla }_{(B_{j},0)}(0,{grad}\ f^{2}))|_{i_{y_{0}}}
\\
&=&-m(\frac{1}{4f^{2}}|{grad}\ f^{2}|^{2}(B_{j},0)+\frac{1}{4b^{2}}%
(B_{j},0)(b^{2})(0,{grad}\ f^{2}))|_{i_{y_{0}}}.
\end{eqnarray*}%
Then
\begin{eqnarray}
\nabla _{B_{j}}\nabla _{B_{j}}\tau (i_{y_{0}}) &=&-\frac{m}{4}\{\frac{1}{%
f^{2}}|{grad}\ f^{2}|^{2}(\nabla _{B_{j}}^{B}B_{j},0)  \notag \\
&&+\frac{1}{2b^{2}f^{2}}(B_{j},0)(b^{2})|{grad}\ f^{2}|^{2}(B_{j},0)  \notag
\\
&&-\frac{1}{2}(\frac{1}{f^{2}}|{grad}\ f^{2}|^{2}+\frac{1}{b^{4}}%
((B_{j},0)(b^{2}))^{2})(0,{grad}\ f^{2}))\}|_{i_{y_{0}}}.
\end{eqnarray}%
Also
\begin{eqnarray}
\nabla _{\nabla _{B_{j}}^{B}B_{j}}\tau (i_{y_{0}}) &=&(-\frac{m}{4f^{2}}|{%
grad}\ f^{2}|^{2}(\nabla _{B_{j}}^{B}B_{j},0)  \notag \\
&&-\frac{m}{4b^{2}}(\nabla _{B_{j}}^{B}B_{j},0)(b^{2})(0,{grad}\
f^{2}))|_{i_{y_{0}}}
\end{eqnarray}%
From the equations (4.1) and (4.2) the rough Laplacian of $\tau (i_{y_{0}})$
is
\begin{eqnarray}
-\Delta \tau (i_{y_{0}}) &=&\frac{m}{4}\sum_{j=1}^{m}\{-\frac{1}{2b^{2}f^{2}}%
(B_{j},0)(b^{2})|{grad}\ f^{2}|^{2}(B_{j},0)  \notag \\
&&+(\frac{1}{2f^{2}}|{grad}\ f^{2}|^{2}-\frac{1}{2b^{4}}%
((B_{j},0)(b^{2}))^{2}  \notag \\
&&+\frac{1}{b^{2}}(\nabla _{B_{j}}^{B}B_{j},0)(b^{2}))(0,{grad}\
f^{2})\}|_{i_{y_{0}}}.
\end{eqnarray}%
On the other hand by using (3.4) it can be seen that
\begin{eqnarray*}
trace_{g_{b}}\overline{R}(di_{y_{0}},\tau (i_{y_{0}}))di_{y_{0}} &=&\{\frac{m%
}{4b^{2}}\sum_{j=1}^{m}(\nabla _{B_{j}}^{B}B_{j})(b^{2})(0,{grad}\ f^{2})
\notag \\
&&+\frac{m^{2}}{8f^{2}}|{grad}\ f^{2}|^{2}(0,{grad}\ f^{2})  \notag \\
&&-\frac{m}{2}\sum_{j=1}^{m}B_{j}(\frac{1}{2b^{2}}B_{j}(b^{2}))(0,{grad}\
f^{2})  \notag \\
&&-\frac{m}{8b^{2}f^{2}}\sum_{j=1}^{m}((B_{j},0)(b^{2})|{grad}\
f^{2}|^{2}(B_{j},0))  \notag \\
&&-\frac{m}{8b^{4}}\sum_{j=1}^{m}(((B_{j},0)(b^{2}))^{2}(0,{grad}\ f^{2}))
\notag \\
&&-\frac{m^{2}}{4}(0,\nabla _{{grad}\ f^{2}}^{F}{grad}\ f^{2})  \notag \\
&&+\frac{m^{2}}{8b^{2}}|{grad}\ f^{2}|^{2}({grad}\ b^{2},0)\}|_{i_{y_{0}}}.
\end{eqnarray*}%
Subtracting the last equation from (4.3) we obtain the bitension field of
the $i_{y_{0}}:(B,g_{B})\rightarrow (_{f}B\times _{b}F,g)$
\begin{eqnarray}
\tau _{2}(i_{y_{0}}) &=&\{-\frac{m^{2}}{8b^{2}}|{grad}\ f^{2}|^{2}({grad}\
b^{2},0)  \notag \\
&&+\frac{m}{2}\sum_{j=1}^{m}(B_{j}(\frac{1}{2b^{2}}B_{j}(b^{2}))(0,{grad}\
f^{2})+\frac{m^{2}}{8}(0,{grad}(|{grad}\ f^{2}|^{2})\}|_{i_{y_{0}}},
\label{eq:4.4}
\end{eqnarray}%
where $\{B_{j}\}_{j=1}^{m}$ is an orthonormal frame on $(B,g_{b})$.

We can assume the normal coordinates about the arbitrary point of $%
(B,g_{b}). $ Therefore the equation (\ref{eq:4.4}) has the form%
\begin{eqnarray*}
\tau _{2}(i_{y_{0}}) &=&\{-\frac{m^{2}}{8b^{2}}|{grad}\ f^{2}|^{2}({grad}\
b^{2},0) \\
&&+\frac{m}{2}\Delta (\ln b)(0,{grad}\ f^{2})+\frac{m^{2}}{8}(0,{grad}(|{grad%
}\ f^{2}|^{2})\}|_{i_{y_{0}}}.
\end{eqnarray*}%
This completes the proof.
\end{proof}

\begin{corollary}
The inclusion map $i_{y_{0}}:(B,g_{B})\rightarrow (_{f}B\times _{b}F,g),$ is
a proper biharmonic map if and only if $b$ is constant and $y_{0}$ is not a
critical point for $f^{2}$ but it is a critical point for $|\func{grad}%
f^{2}|^{2}$.
\end{corollary}

\begin{corollary}
Each and every inclusion map $i_{y}:(B,g_{b})\rightarrow (_{f}B\times
_{b}F,g),$ $y\in F,$ is a proper biharmonic map if and only if the function $%
b$ is constant and \ $\func{grad}f^{2}$ is a non-zero constant norm vector
field.
\end{corollary}

Notice that the constancy of $b$ reduces the doubly warped product manifold $%
_{f}B\times _{b}F$ to the warped product manifold. So we have

\begin{corollary}
Let $(_{f}B\times _{b}F,g)$ be a doubly warped product manifold with
non-constant warping functions $b$ and $f$. Then the inclusion map of the
manifold $(B,g_{B})$ into the doubly warped product manifold $(_{f}B\times
_{b}F,g)$ is never a proper biharmonic map.
\end{corollary}

\begin{remark}
When $f=1$, the doubly warped product manifold $_{f}B\times _{b}F$ becomes a
warped product $B\times _{b}F$ . Since the inclusion map $%
i_{y_{0}}:(B,g_{B})\rightarrow (B\times _{b}F,g)$ of $B$ at the level $%
y_{0}\in F$ is always totally geodesic, then it's harmonic for any warping
function $b\in C^{\infty }(B).$ So in the case of warped product
biharmonicity of the inclusion map $i_{y_{0}}:(B,g_{B})\rightarrow (B\times
_{b}F,g)$ is trivial.
\end{remark}

Now let us consider the iclusion map $i_{x_{0}}:(F,g_{F})\rightarrow
(_{f}B\times _{b}F,g),$ $y\rightarrow (x_{0},y),$ of $(F,g_{F})$ into the
doubly warped product manifold $(_{f}B\times _{b}F,g)$ where $x_{0}\in B.$
We have,

\begin{theorem}
The bitension field of the inclusion map $i_{x_{0}}:(F,g_{F})\rightarrow
(_{f}B\times _{b}F,g)$ is given by
\begin{eqnarray*}
\tau _{2}(i_{x_{0}}) &=&\{\frac{n^{2}}{8}({grad}(|{grad}\ b^{2}|^{2}),0)+%
\frac{n}{2}\Delta (\ln b)({grad}\ b^{2},0) \\
&&-\frac{n^{2}}{8f^{2}}|\func{grad}b^{2}|^{2}(0,\func{grad}%
f^{2})\}|_{i_{x_{0}}}.
\end{eqnarray*}
\end{theorem}

\begin{proof}
Let $\{F_{r}\}_{r=1}^{n}$ be an orthonormal frame on $(F,g_{f})$. One can
easily see from (2.1) the tension field of $i_{x_{0}}$ is%
\begin{eqnarray}
\tau (i_{x_{0}}) &=&trace_{g_{f}}\nabla di_{x_{0}}  \notag \\
&=&\sum_{r=1}^{n}\{\nabla _{F_{r}}di_{x_{0}}(F_{r})-di_{x_{0}}(\nabla
_{F_{r}}^{F}F_{r})\}  \notag \\
&=&\sum_{r=1}^{n}\{\overline{\nabla }_{(0,F_{r})}(0,F_{r})-(0,\nabla
_{F_{r}}^{F}F_{r})\}  \notag \\
&=&-\frac{n}{2}({grad}\ b^{2},0)|_{i_{x_{0}}}.
\end{eqnarray}%
Notice now that, from (4.5), $i_{x_{0}}$ is harmonic if and only if $({grad}%
\ b^{2})|_{x_{0}}=0$. In order to express the bitension field of $i_{x_{0}}$%
, we firstly compute the rough Laplacian. Since%
\begin{eqnarray*}
\nabla _{F_{r}}\tau (i_{x_{0}}) &=&-\frac{n}{2}\nabla _{F_{r}}({grad}\
b^{2},0)|_{i_{x_{0}}} \\
&=&-\frac{n}{2}(\overline{\nabla }_{(0,F_{r})}({grad}\ b^{2},0))|_{i_{x_{0}}}
\\
&=&-\frac{n}{4}\{\frac{1}{f^{2}}(0,F_{r})(f^{2})({grad}\ b^{2},0)+\frac{1}{%
b^{2}}|{grad}\ b^{2}|^{2}(0,F_{r})\}|_{i_{x_{0}}}
\end{eqnarray*}%
then
\begin{eqnarray}
\nabla _{F_{r}}\nabla _{F_{r}}\tau (i_{x_{0}}) &=&-\frac{n}{4}\{(\frac{1}{%
2f^{4}}((0,F_{r})(f^{2}))^{2}-\frac{1}{2b^{2}}|{grad}\ b^{2}|^{2})({grad}\
b^{2},0)  \notag \\
&&+\frac{1}{2b^{2}f^{2}}(0,F_{r})(f^{2})|{grad}\ b^{2}|^{2}(0,F_{r})  \notag
\\
&&+\frac{1}{b^{2}}|{grad}\ b^{2}|^{2}(0,\nabla
_{F_{r}}^{F}F_{r})\}|_{i_{x_{0}}}.
\end{eqnarray}%
Moreover we get%
\begin{eqnarray}
\nabla _{\nabla _{F_{r}}^{F}F_{r}}\tau (i_{x_{0}}) &=&-\frac{n}{4}\{\frac{1}{%
f^{2}}(0,\nabla _{F_{r}}^{F}F_{r})(f^{2})({grad}\ b^{2},0)  \notag \\
&&+\frac{1}{b^{2}}|{grad}\ b^{2}|^{2}(0,\nabla
_{F_{r}}^{F}F_{r})\}|_{i_{x_{0}}}.
\end{eqnarray}%
Therefore from (4.6) and (4.7)
\begin{eqnarray}
-\Delta \tau (i_{x_{0}}) &=&\frac{n}{4}\sum_{r=1}^{n}\{-\frac{1}{2f^{4}}%
((0,F_{r})(f^{2}))^{2}({grad}\ b^{2},0)  \notag \\
&&+\frac{1}{2b^{2}}|{grad}\ b^{2}|^{2}({grad}\ b^{2},0)  \notag \\
&&+\frac{1}{f^{2}}(0,\nabla _{F_{r}}^{F}F_{r})(f^{2})({grad}\ b^{2},0)
\notag \\
&&-\frac{1}{2b^{2}f^{2}}(0,F_{r})(f^{2})|{grad}\
b^{2}|^{2}(0,F_{r})\}|_{i_{x_{0}}}.
\end{eqnarray}%
Next by a straightforward computation we have
\begin{eqnarray}
trace_{g_{f}}\overline{R}(di_{x_{0}},\tau (i_{x_{0}}))di_{x_{0}} &=&\{\frac{n%
}{4f^{2}}\sum_{r=1}^{n}(\nabla _{F_{r}}^{F}F_{r})(f^{2})({grad}\ b^{2},0)
\notag \\
&&-\frac{n}{8f^{4}}\sum_{r=1}^{n}((0,F_{r})(f^{2}))^{2}({grad}\ b^{2},0)
\notag \\
&&+\frac{n^{2}}{8b^{2}}|\func{grad}b^{2}|^{2}({grad}\ b^{2},0)  \notag \\
&&-\frac{n}{2}\sum_{r=1}^{n}(F_{r}(\frac{1}{2f^{2}}F_{r}(f^{2}))({grad}\
b^{2},0)  \notag \\
&&-\frac{n}{8b^{2}f^{2}}(0,F_{r})(f^{2})|{grad}\ b^{2}|^{2}(0,F_{r})  \notag
\\
&&-\frac{n}{4}(\nabla _{{grad}\ b^{2}}^{B}{grad}\ b^{2},0)  \notag \\
&&+\frac{n^{2}}{8f^{2}}|{grad}\ b^{2}|^{2}(0,{grad}\ f^{2})\}|_{i_{x_{0}}}.
\end{eqnarray}%
So from (4.8) and (4.9) the bitension field of $i_{x_{0}}$ is
\begin{eqnarray}
\tau _{2}(i_{x_{0}}) &=&\{\frac{n^{2}}{8}({grad}(|{grad}\ b^{2}|^{2}),0)+%
\frac{n}{2}\sum_{r=1}^{n}F_{r}(\frac{1}{2f^{2}}F_{r}(f^{2}))({grad}\ b^{2},0)
\notag \\
&&-\frac{n^{2}}{8f^{2}}|\func{grad}b^{2}|^{2}(0,\func{grad}%
f^{2})\}|_{i_{x_{0}}}.
\end{eqnarray}

Then by using the normal coordinates about the each points of the manifold $%
(F,g_{F}),$ we get
\begin{eqnarray*}
\tau _{2}(i_{x_{0}}) &=&\{\frac{n^{2}}{8}({grad}(|{grad}\ b^{2}|^{2}),0)+%
\frac{n}{2}\Delta (\ln f)({grad}\ b^{2},0) \\
&&-\frac{n^{2}}{8f^{2}}|\func{grad}b^{2}|^{2}(0,\func{grad}%
f^{2})\}|_{i_{x_{0}}},
\end{eqnarray*}%
and we conclude.
\end{proof}

\begin{corollary}
The inclusion map $i_{x_{0}}:(F,g_{F})\rightarrow (_{f}B\times _{b}F,g)$ $,$
is a proper biharmonic map if and only if$\ x_{0}$ is not a critical point
for $b^{2}$ but it is a critical point for $|\func{grad}b^{2}|^{2}$ and the
warping function $f$ is constant.
\end{corollary}

\begin{corollary}
Each and every inclusion map $i_{x}:(F,g_{F})\rightarrow (_{f}B\times
_{b}F,g),$ $x\in B,$ is a proper biharmonic map if and only if \ $\func{grad}%
b^{2}$ is a non-zero constant norm vector field and the function $f$ is
constant.
\end{corollary}

Notice that the constancy of $f$ reduces the doubly warped product manifold $%
_{f}B\times _{b}F$ to the warped product manifold. So we have

\begin{corollary}
Let $(_{f}B\times _{b}F,g)$ be a doubly warped product manifold with
non-constant warping functions $b$ and $f$. Then the inclusion map of the
manifold $(F,g_{F})$ into the doubly warped product manifold $(_{f}B\times
_{b}F,g)$ is never a proper biharmonic map.
\end{corollary}

\begin{remark}
If $f=1$, $_{f}B\times _{b}F$ becomes a warped product manifold and
we obtain the corollaries 3.3, 3.4 and 3.5 in [4].
\end{remark}

\section{Product maps}

\setcounter{equation}{0} \renewcommand{\theequation}{5.\arabic{equation}}

Let $I_{B}:B\rightarrow B$ be the identity map on $B$ and $\varphi
:F\rightarrow F$ be a harmonic map. Obviuosly $\Psi =I_{B}\times \varphi
:B\times F\rightarrow B\times F$ is a harmonic map. Now suppose the product
manifold $B\times F$ (either as domain or codomain) with the doubly warped
product metric tensor $g=f^{2}g_{B}\oplus b^{2}g_{F}.$ In this case the
product map is no longer harmonic. Therefore in this section we shall obtain
some results for maps of product type to be proper biharmonic.

Firstly let us consider the product map%
\begin{equation*}
\overline{\Psi }=\overline{I_{B}\times \varphi }:\,_{f}B\times
_{b}F\rightarrow B\times F.
\end{equation*}

\begin{theorem}
Let $(B,g_{B})$ and $(F,g_{F})$ be Riemannian manifolds of dimensions $m$
and $n,$ respectively and let $b:B\rightarrow (0,\infty )$ and $%
f:F\rightarrow (0,\infty )$ be smooth functions. Suppose that $\varphi
:F\rightarrow F$ is a harmonic map. Then $\overline{\Psi }=\overline{%
I_{B}\times \varphi }:\,_{f}B\times _{b}F\rightarrow B\times F$ is a proper
biharmonic map if and only if $b$ is a non-constant solution of
\begin{eqnarray}
\,\,\,\,\,\,\,\,0=\frac{1}{f^{2}}trace_{g_{b}}\nabla ^{2}{grad}\ln b+\frac{1%
}{f^{2}}Ricc^{B}({grad}\ln b)+\frac{n}{2}{grad}(|{grad}\ln b|^{2})
\end{eqnarray}%
and $f$ is a non-constant solution of%
\begin{eqnarray}
0=-\frac{1}{b^{2}}J_{\varphi }(d\varphi ({grad}\ln f))+\frac{m}{2}{grad}%
(|d\varphi ({grad}\ln f)|^{2})
\end{eqnarray}
\end{theorem}

\begin{proof}
Let $\{B_{j}\}_{j=1}^{m}$ and $\{F_{r}\}_{r=1}^{n}$ be local orthonormal
frames on $(B,g_{B})$ and $(F,g_{F}),$ respectively. Then $\{\frac{1}{f}%
(B_{j},0),\frac{1}{b}(0,F_{r})\}_{j,r=1}^{m,n}$ is an local orthonormal
frame on the doubly warped product manifold $_{f}B\times _{b}F.$ In order to
obtain bitension field of $\overline{\Psi }$, firstly we will compute the
tension field of $\overline{\Psi }.$ Since $\varphi $ is harmonic,%
\begin{eqnarray}
\tau (\overline{\Psi }) &=&trace_{g}\nabla d\overline{\Psi }  \notag \\
&=&\frac{1}{f^{2}}\sum_{j=1}^{m}\nabla d\overline{\Psi }%
((B_{j},0),(B_{j},0))+\frac{1}{b^{2}}\sum_{r=1}^{n}\nabla d\overline{\Psi }%
((0,F_{r}),(0,F_{r}))  \notag \\
&=&\frac{1}{f^{2}}\sum_{j=1}^{m}\{\nabla _{(B_{j},0)}^{\overline{\Psi }}d%
\overline{\Psi }(B_{j},0)-d\overline{\Psi }(\overline{\nabla }%
_{(B_{j},0)}(B_{j},0))\}  \notag \\
&&+\frac{1}{b^{2}}\sum_{r=1}^{n}\{\nabla _{(0,F_{r})}^{\overline{\Psi }}d%
\overline{\Psi }(0,F_{r})-d\overline{\Psi }(\overline{\nabla }%
_{(0,F_{r})}(0,F_{r}))\}  \notag \\
&=&\frac{1}{f^{2}}\sum_{j=1}^{m}\{\nabla _{(B_{j},0)}(B_{j},0)-d\overline{%
\Psi }(\overline{\nabla }_{(B_{j},0)}(B_{j},0))\}  \notag \\
&&+\frac{1}{b^{2}}\sum_{r=1}^{n}\{\nabla _{d\overline{\Psi }(0,F_{r})}d%
\overline{\Psi }(0,F_{r})-d\overline{\Psi }(\overline{\nabla }%
_{(0,F_{r})}(0,F_{r}))\}  \notag \\
&=&n({grad}\ln b,0)+m(0,d\varphi ({grad}\ln f)).
\end{eqnarray}
By a straightforward calculation we have%
\begin{eqnarray}
-\Delta \tau (\overline{\Psi }) &=&trace_{g}\nabla ^{2}\tau (\overline{\Psi }%
)  \notag \\
&=&\sum_{j=1}^{m}\{\nabla _{\frac{1}{f}(B_{j},0)}^{\overline{\Psi }}\nabla _{%
\frac{1}{f}(B_{j},0)}^{\overline{\Psi }}\tau (\overline{\Psi })-\nabla _{%
\overline{\nabla }_{\frac{1}{f}(B_{j},0)}\frac{1}{f}(B_{j},0)}^{\overline{%
\Psi }}\tau (\overline{\Psi })\}  \notag \\
&&+\sum_{r=1}^{n}\{\nabla _{\frac{1}{b}(0,F_{r})}^{\overline{\Psi }}\nabla _{%
\frac{1}{b}(0,F_{r})}^{\overline{\Psi }}\tau (\overline{\Psi })-\nabla _{%
\overline{\nabla }_{\frac{1}{b}(0,F_{r})}\frac{1}{b}(0,F_{r})}^{\overline{%
\Psi }}\tau (\overline{\Psi })\}  \notag \\
&=&(\frac{n}{f^{2}}trace_{g_{b}}\nabla ^{2}{grad}\ln b+n^{2}\nabla _{{grad}%
\ln b}{grad}\ln b,0)  \notag \\
&&+(0,\frac{m}{b^{2}}trace_{g_{f}}\nabla ^{2}(d\varphi ({grad}\ln
f))+m^{2}\nabla _{d\varphi ({grad}\ln f)}d\varphi ({grad}\ln f)).
\end{eqnarray}%
Also by using the usual definition of curvature tensor field on $B\times F,$
one can easily see that%
\begin{eqnarray}
trace_{g}R(d\overline{\Psi },\tau (\overline{\Psi }))d\overline{\Psi } &=&-%
\frac{n}{f^{2}}(Ricc^{B}({grad}\ln b),0)  \notag \\
&&+\frac{m}{b^{2}}(0,trace_{g_{f}}R^{F}(d\varphi ,d\varphi ({grad}\ln
f))d\varphi ).
\end{eqnarray}%
Finally bitension field of $\overline{\Psi }$ is
\begin{eqnarray}
\tau _{2}(\overline{\Psi }) &=&n(\frac{1}{f^{2}}trace_{g_{b}}\nabla ^{2}{grad%
}\ln b+\frac{1}{f^{2}}Ricc^{B}({grad}\ln b)+\frac{n}{2}{grad}(|{grad}\ln
b|^{2}),0)  \notag \\
&&+m(0,\frac{1}{b^{2}}trace_{g_{f}}\nabla ^{2}(d\varphi ({grad}\ln f)))
\notag \\
&&-m(0,\frac{1}{b^{2}}trace_{g_{f}}R^{F}(d\varphi ,d\varphi ({grad}\ln
f))d\varphi )  \notag \\
&&+m^{2}(0,\frac{1}{2}{grad}(|d\varphi ({grad}\ln f)|^{2}))
\end{eqnarray}%
and we conclude.
\end{proof}

\bigskip

We shall now investigate the biharmonicity of the projection $\bar{\pi}%
:_{f}B\times _{b}F\rightarrow B$ onto the first factor. By a straightforward
calculation, we get $\tau (\bar{\pi})=n({grad}\ln b)\circ \bar{\pi},$ and
the bitension field of $\bar{\pi}$ is%
\begin{eqnarray}
\tau _{2}(\bar{\pi}) &=&\frac{n}{f^{2}}trace_{g_{b}}\nabla ^{2}{grad}\ln b
\notag \\
&&+\frac{n}{f^{2}}Ricc^{B}({grad}\ln b)+\frac{n^{2}}{2}{grad}(|{grad}\ln
b|^{2}).
\end{eqnarray}%
By using the bitension field of the projection $\bar{\pi}$, the biharmonic
equation of the product map $\overline{\Psi }=\overline{I_{B}\times \varphi }%
:\,_{f}B\times _{b}F\rightarrow B\times F$ has the expression
\begin{equation}
\tau _{2}(\overline{\Psi })=(\tau _{2}(\bar{\pi}),\Lambda (d\varphi ({grad}%
\ln f))=0,
\end{equation}%
where
\begin{equation}
\ \Lambda (d\varphi ({grad}\ln f))=-\frac{m}{b^{2}}J_{\varphi }(d\varphi ({%
grad}\ln f)+\frac{m^{2}}{2}{grad}(|d\varphi ({grad}\ln f)|^{2}).
\end{equation}%
So we have

\begin{corollary}
The product map $\overline{\Psi }=\overline{I_{B}\times \varphi }%
:\,_{f}B\times _{b}F\rightarrow B\times F$ is a proper biharmonic map if and
only if $\ \Lambda (d\varphi ({grad}\ln f))=0$ and the projection $\bar{\pi}$
is a proper biharmonic map.
\end{corollary}

If $f=1$ or $\varphi :F\rightarrow F$ is a non-zero constant map the
biharmonicity of the product map $\overline{\Psi }=\overline{I_{B}\times
\varphi }:\,_{f}B\times _{b}F\rightarrow B\times F$ \ and the biharmonicity
of the projection $\bar{\pi}:_{f}B\times _{b}F\rightarrow B$ onto the first
factor coincide. In the case of $f=1$, we have the same results obtained in
[4] for warped product manifolds.

\bigskip

Now let us consider the product map $\widetilde{\Psi}=\widetilde{\varphi
\times I_{F}}:\,_{f}B\times _{b}F\rightarrow B\times F$ of the harmonic map $%
\varphi :B\rightarrow B$ and the identitiy map $I_{F}:F\rightarrow F$.

\begin{theorem}
Let $(B,g_{B})$ and $(F,g_{F})$ be Riemannian manifolds of dimensions $m$
and $n,$ respectively and let $b:B\rightarrow (0,\infty )$and $%
f:F\rightarrow (0,\infty )$ be smooth functions. Suppose that $\varphi
:B\rightarrow B$ is a harmonic map. Then $\widetilde{\Psi }=\widetilde{%
\varphi \times I_{F}}:\,_{f}B\times _{b}F\rightarrow B\times F$ is a proper
biharmonic map if and only if $b$ is a non-constant solution of
\begin{equation}
0=-\frac{1}{f^{2}}J_{\varphi }(d\varphi ({grad}\ln b))+\frac{n}{2}{grad}%
(|d\varphi ({grad}\ln b)|^{2})
\end{equation}%
and $f$ is a non-constant solution of%
\begin{equation}
0=\frac{1}{b^{2}}trace_{g_{f}}\nabla ^{2}{grad}\ln f+\frac{1}{b^{2}}Ricc^{F}(%
{grad}\ln f)+\frac{m}{2}{grad}(|{grad}\ln f|^{2}).
\end{equation}
\end{theorem}

\bigskip

By a similar discussion, carried out to set up the relation between the
biharmonicity of product map $\overline{\Psi }=\overline{I_{B}\times \varphi
}:\,_{f}B\times _{b}F\rightarrow B\times F$ \ and the biharmonicity of \ the
projection $\overline{\pi}:_{f}B\times _{b}F\rightarrow B$ onto the first
factor, the bitension field of the projection $\widetilde{\sigma}%
:_{f}B\times _{b}F\rightarrow F$ onto the second factor is related to the
bitension field of $\widetilde{\Psi }=\widetilde{\varphi \times I_{F}}%
:\,_{f}B\times _{b}F\rightarrow B\times F$ as follows:

By computing the second fundamental form of the projection $\widetilde{%
\sigma },$\ we get $\tau (\widetilde{\sigma })=m({grad}\ln f)\circ
\widetilde{\sigma },$ and the bitension field of $\widetilde{\sigma }$ is%
\begin{eqnarray}
\tau _{2}(\widetilde{\sigma }) &=&\frac{m}{b^{2}}trace_{g_{f}}\nabla ^{2}{%
grad}\ln f  \notag \\
&&+\frac{m}{b^{2}}Ricc^{F}({grad}\ln f)+\frac{m^{2}}{2}{grad}(|{grad}\ln
f|^{2}).
\end{eqnarray}%
Now by using the bitension field of the projection $\widetilde{\sigma }$,
the biharmonic equation of the product map $\widetilde{\Psi }=\widetilde{%
\varphi \times I_{F}}:\,_{f}B\times _{b}F\rightarrow B\times F$ is
\begin{equation}
\tau _{2}(\widetilde{\Psi })=(\Omega (d\varphi ({grad}\ln b),\tau _{2}(%
\widetilde{\sigma }))=0,
\end{equation}%
where
\begin{equation}
\ \Omega (d\varphi ({grad}\ln b))=-\frac{n}{f^{2}}J_{\varphi }(d\varphi ({%
grad}\ln b))+\frac{n^{2}}{2}{grad}(|d\varphi ({grad}\ln b)|^{2}).
\end{equation}

\begin{corollary}
The product map $\widetilde{\Psi }=\widetilde{\varphi \times I_{F}}%
:\,_{f}B\times _{b}F\rightarrow B\times F$ is a proper biharmonic map if and
only if $\ \Omega (d\varphi ({grad}\ln b))=0$ and the projection $\widetilde{%
\sigma}$ is a proper biharmonic map.
\end{corollary}

Note that especially if $\varphi :B\rightarrow B$ is the identity map then
the bitension field of of the product map $\widetilde{\Psi }=\widetilde{%
\varphi \times I_{F}}:\,_{f}B\times _{b}F\rightarrow B\times F$ has the
expression $\tau _{2}(\widetilde{\Psi })=(\tau _{2}(\overline{\pi }),\tau
_{2}(\widetilde{\sigma })).$

\bigskip

Consider now the case of the product map $\widehat{\Psi }=\widehat{%
I_{B}\times \varphi }:B\times F\,\rightarrow _{f}B\times _{b}F,$ that is the
case of the product metric on the codomain is doubly warped. We will see
that the energy density of the harmonic map $\varphi :F\rightarrow F$ has an
important role for the biharmonicity of the product map $\widehat{\Psi }=%
\widehat{I_{B}\times \varphi }.$ We have

\begin{theorem}
Let $(B,g_{B})$ and $(F,g_{F})$ be Riemannian manifolds of dimensions $m$
and $n,$ respectively and let $b:B\rightarrow (0,\infty )$ and $%
f:F\rightarrow (0,\infty )$ be smooth functions. Suppose $I_{B}:B\rightarrow
B$ is the identity map and $\varphi :F\rightarrow F$ is a harmonic map. Then
the product map $\widehat{\Psi }=\widehat{I_{B}\times \varphi }:B\times
F\,\rightarrow _{f}B\times _{b}F$ is a proper biharmonic map if and only if%
\begin{eqnarray}
0 &=&e(\varphi )\{-trace_{g_{b}}\nabla ^{2}(\func{grad}b^{2})-Ricc^{B}(\func{%
grad}b^{2})+\frac{e(\varphi )}{2}\func{grad}(|\func{grad}b^{2}|^{2})\}
\notag \\
&&+\frac{m}{4}\{\frac{e(\varphi )}{f^{2}}-\frac{m}{2b^{2}}\}|\func{grad}%
f^{2}|^{2}\func{grad}b^{2}  \notag \\
&&+\frac{m}{4b^{2}}\{d\varphi (\Delta (f^{2}))\}\func{grad}b^{2}  \notag \\
&&+\frac{e(\varphi )}{2}d\varphi (\Delta (\ln f))\func{grad}b^{2}
\end{eqnarray}
and
\begin{eqnarray}
0&=&\frac{m}{2}\{-J_{\varphi }(\func{grad}f^{2})+\frac{m}{4}\func{grad}(|%
\func{grad}f^{2}|^{2})\}  \notag \\
&&-\frac{e(\varphi )}{2}\{\frac{e(\varphi )}{f^{2}}-\frac{m}{2b^{2}}\}|\func{%
grad}b^{2}|^{2}\func{grad}f^{2}  \notag \\
&&-\frac{|\func{grad}b^{2}|^{2}}{2b^{2}}d\varphi (\func{grad}e(\varphi ))+%
\frac{e(\varphi )}{2f^{2}}\Delta (b^{2})\func{grad}f^{2}  \notag \\
&&+\frac{m}{4}\Delta (\ln b)\func{grad}f^{2}.
\end{eqnarray}
\end{theorem}


\begin{thebibliography}{99}
\bibitem{Baird} P. Baird, D. Kamissoko, On constracting biharmonic maps and
metrics, Ann. Global Anal. Geom. 23 (2003) 65-75.

\bibitem{Wood} P. Baird, J.C. Wood, Harmonic Morphisms between Riemannian
Manifolds, Oxford Science Publications, 2003.

\bibitem{Balmus} A. Balmu\c{s}, Biharmonic properties and conformal changes,
An. \c{S}tiin\c{t}. Univ. Al.I. Cuza Ia\c{s}i Mat. (N.S.) 50 (2004) 361-372.

\bibitem{Balmus-mont} A. Balmu\c{s}, S. Montaldo, C. Oniciuc, Biharmonic
maps between warped product manifolds, J. Geom. Phys. 57 (2007) 449-466.

\bibitem{Beem} J.K. Beem, P.E. Ehrlich, K.L. Easley, Global Lorentzian
Geometry, Marcel Dekker, Second Edition, New York, 1996.

\bibitem{Bishop} R.L. Bishop, B. O'Neill, Manifolds of negative curvature,
Trans. Amer. Math. Soc. 145 (1969) 1-49.

\bibitem{Dob} F. Dobarro, B. \"{U}nal, Curvature of multiply warped
products, J. Geom. Phys. 55 (2005) 75-106.

\bibitem{Eells} J. Eells, J.H. Sampson, Harmonic mappings of Riemannian
manifolds, Amer. J. Math. 86 (1964) 109-160.

\bibitem{Luc} J. Eells, L. Lemaire, Selected topics in harmonic maps, Conf.
Board. Math. Sci. 50 (1983) 109-160.

\bibitem{Kupeli} M. Fern\'{a}ndez-L\'{o}pez, E. Garc\'{\i}a-R\'{\i}o, D.N.
Kupeli, B. \"{U}nal, A curvature condition for a twisted product to be a
warped product, Manuscripta math. 106 (2001) 213-217.

\bibitem{Fuglede} B. Fuglede, Harmonic morphisms between Riemannian
manifolds, Ann. Inst. Fourier (Grenoble) 28 (2) (1978) 107-144.

\bibitem{Jiang} G.Y. Jiang, 2-harmonic isometric immersions between
Riemannian manifolds, Chinese Ann. Math. Ser. A 7 (1986) 130-144.

\bibitem{Jiang2} G.Y. Jiang, 2-harmonic maps and their first and second
variation formulas, Chinese Ann. Math. Ser. A 7 (1986) 389-402.

\bibitem{O'Neill} B. O'Neill, Semi-Riemannian Geometry with Applications to
Relativity, New York, Academic Press, 1983.

\bibitem{Oniciuc} C. Oniciuc, Biharmonic maps between Riemannian manifolds,
An. \c{S}tiin\c{t}. Univ. Al.I. Cuza Ia\c{s}i Mat. (N.S.) 48 (2002) 237-248.

\bibitem{Ou} Y.-L. Ou, p-harmonic morphisms, biharmonic morphisms and
nonharmonic biharmonic maps, J. Geom. Phys. 56 (2006) 358-374.

\bibitem{Reckziegel} R. Ponge, H. Reckziegel, Twisted Products in
Pseudo-Riemannian Geometry, Geom. Dedicata 48 (1993) 15-25.

\bibitem{Unal1} B. \"{U}nal, Multiply warped products, J. Geom. Phys. 34
(2000) 287-301.

\bibitem{Unal2} B. \"{U}nal, Doubly warped products, Diff. Geom. and its
Applications, 15 (2001) 253-263.
\end{thebibliography}
\end{document}